\documentclass{article}
\usepackage{amsmath}
\usepackage{amsthm}
\usepackage{amstext}
\usepackage{amsfonts}
\usepackage{amssymb}

\usepackage[all,cmtip]{xy}
\usepackage{amssymb}
\usepackage[justification=centering]{caption}
 \usepackage{paralist}
  \usepackage{graphics}
  \usepackage{epsfig}
\usepackage{graphicx}
\usepackage{epstopdf}
 \usepackage[colorlinks=true]{hyperref}
\usepackage{mathrsfs}
\usepackage{tikz}
\usetikzlibrary{arrows,shapes,positioning}
\usetikzlibrary{decorations.markings}
\tikzstyle arrowstyle=[scale=1.5]
\tikzstyle directed=[postaction={decorate,decoration={markings,mark=at position .5 with {\arrow[arrowstyle]{stealth}}}}]
\tikzstyle reverse directed=[postaction={decorate,decoration={markings,mark=at position .5 with {\arrowreversed[arrowstyle]{stealth};}}}]
\usetikzlibrary{calc}

\newtheorem{theorem}{Theorem}[section]
\newtheorem{lemma}[theorem]{Lemma}
\newtheorem{proposition}[theorem]{Proposition}
\theoremstyle{definition}
\newtheorem{definition}[theorem]{Definition}
\newtheorem{remark}[theorem]{Remark}
\newtheorem*{rmk}{Remark}
\newtheorem*{pro}{Proposition}
\numberwithin{equation}{section}
\theoremstyle{corollary}

\theoremstyle{conjecture}
\newtheorem*{conjecture}{Weak Palis Conjecture}
\theoremstyle{main}
\newtheorem*{main}{Main Theorem}
\theoremstyle{claim}
\newtheorem*{claim}{Claim}
\hypersetup{urlcolor=blue, citecolor=red}
\begin{document}

\title{A remark on the central model method for the weak Palis conjecture of higher dimensional singular flows}
\author{Qianying Xiao and Yiwei Zhang}
\maketitle
{\footnotesize
 \centerline{School of Mathematical Sciences, Peking University}
   \centerline{Beijing, 100871, P. R. China}
   \centerline{xiaoqianying@math.pku.edu.cn}
   \medskip
   \centerline{School of Mathematics and Statistics}
   \centerline{Center for Mathematical Sciences}
   \centerline{Hubei Key Laboratory of Engineering Modeling and Scientific Computing}
   \centerline{HuaZhong University of Science and Technology}
   \centerline{1037 Luoyu Road, Wuhan, 430074, P. R. China}
   \centerline{yiweizhang@hust.edu.cn}
} 

\bigskip

\begin{abstract}
For a generic vector field robustly without horseshoes, and an aperiodic chain recurrent class with singularities whose saddle values have different signs, the extended rescaled Poincar\'e map is associated with a central model. We estimate such central model and show it must have chain recurrent central segments over the singularities. This obstructs the application of central model to create horseshoes, and indicates that, differing from $C^1$ diffeomorphisms, solo using central model method is insufficient as a strategy to prove weak Palis conjecture for higher dimensional ($\geq 4$) singular flows.
Our computation is actually based on simplified way of addressing blowup construction. As a byproduct, we are applicable to directly compute the extended rescaled Poincar\'e map upto second order derivatives, which we believe has its independent interests.
\end{abstract}
\section{Introduction}
One goal of modern differential dynamical system theory is to classify dynamical behaviours for most systems. Under this framework, Palis \cite{Pa00,Pa08} proposed several famous density conjectures in the late 20th century, and they attract great interests afterwards \cite{B,BDV,CP,Ts}. One of the density conjectures concerns two extremely different kinds of systems. Namely Morse-Smale systems which are simple in that they have robustly finite periodic orbits; systems displaying horseshoes which are chaotic because they have robustly infinite periodic orbits. To be more precise, the conjecture is stated as:
\begin{conjecture}
The collection of Morse-Smale systems and horseshoe systems is $C^r$ open and dense in certain space of dynamical systems with $r\geq 1$.
\end{conjecture}
Here the systems can be interpreted continuous or discrete. There are a number of attempts to prove this conjecture. For $C^1$ diffeomorphisms and $C^1$ nonsingular flow, the conjecture is proved positively in \cite{BGW07,Cr10,PS00} and\cite{AH,Xiao} respectively. While the progress of the singular flow case is comparatively slow. The core problem is to eliminate generic singular aperiodic chain recurrent classes. The presence of singularities adds huge difficulties, for example, the matching between the hyperbolic splittings of singularities and those of of nearby periodic points. Until recently, Gan-Yang \cite{GY} eventually prove the conjecture holds for three dimensional $C^1$ singular flows. Besides the idea of generalized linear Poincar\'e flow introduced in \cite{LGW}, the dimension restriction is crucial in their proof. In fact, the chain recurrent classes under their consideration must be Lyapunov stable(with respect to the flow or its inverse) so that they are able to construct singular return map to deduce a contradiction. While as the dimension increases, the discussions are becoming much more complicated (See for example \cite{Zheng} and the references therein for detailed arguments). In fact, the difficulties lie in ruling out generic aperiodic singular chain recurrent classes that are not Lyapunov stable. For example, there might exist aperiodic chain recurrent classes which are partially hyperbolic with one dimensional center(with respect to linear Poincar\'e flow) and therefore are not singular hyperbolic. The motivation of this paper is to eliminate these singular aperiodic chain recurrent classes.

To this end, a plausible machinery is by the central model method. This method was firstly introduced by Crovisier \cite{Cr10}, and has successfully dealt with the neutral one dimensional center to create horseshoes in proving weak Palis conjecture for $C^1$ diffeomorphisms. To be more precise, a central model is a pair $ (\hat{K},\hat{f}) $, where $ \hat{K} $ is a compact metric space and $ \hat{f} $ is a continuous map from $\hat{K}\times[0,1]  $ to $ \hat{K}\times[0,+\infty) $ such that:
\begin{itemize}
\item $ \hat{f}(\hat{K}\times\{0\})= \hat{K}\times\{0\}$, and $ \hat{f} $ is a local homeomorphism in a small neighborhood of $ \hat{K}\times\{0\} $;
\item $ \hat{f} $ is a skew-product: there exist two maps $ \hat{f}_1:\hat{K}\rightarrow\hat{K} $ and $ \hat{f}_2:\hat{K}\times[0,1]\rightarrow[0,+\infty) $ such that for any $ (x,t)\in \hat{K}\times[0,1]$, one has $\hat{f}(x,t)=(\hat{f}_1(x),\hat{f}_2(x,t))$.
\end{itemize}

Suppose the base $\hat{K}\times \{0\} $ is chain transitive. For $\hat{x}\in \hat{K}$ and $0<a<1$, the segment $\{\hat{x}\}\times [0,a] $ is called a \emph{chain recurrent central segment} if it is in the same chain recurrent class as $\hat{K}\times \{0\} $. The birth of horseshoes by central model is based on a dichotomy with the flavor of Conley theory. Namely, either the base is a chain recurrent class and therefore admits arbitrarily small attracting/repelling neighborhoods, or there exists a chain recurrent central segment.

Differing from the nonsingular flows, one can not apply central model directly to the Poincar\'e maps, because the sizes of the domains of the Poincar\'e maps tend to zero nearby the singularities. Instead, we are inspired by the ideas of Gan-Yang \cite{GY} and consider rescaled Poincar\'e maps. In fact, the idea of rescaling by the flow speed dates back to Liao \cite{Liao74,Liao89,Liaobk}. Let us recall the definition quickly. Let $X$ be a $C^1$ vector field on a compact manifold $M$. The flow of $X$ is denoted by $\phi_t$. Given a regular point $x$, $\langle X(x)\rangle^\perp$ is denoted by $\mathcal{N}_x$. For $T>0$ and $0<r\ll 1$, let us denote $\mathcal{N}_x(r)=\{v\in \mathcal{N}_x:\Vert v\Vert <r \}$. The rescaled Poincar\'e map $\mathcal{P}^*_{T,x} :\mathcal{N}_x(r)\rightarrow \mathcal{N}_{\phi_T(x)}$ is defined as:
\begin{align*}
\mathcal{P}^*_{T,x}(v)= \frac{\exp^{-1}_{\phi_T(x)}\circ P_{T,x}\circ \exp_x(\Vert X(x)v\Vert)}{\Vert X(\phi_T(x))\Vert}
\end{align*}
here $P_{T,x}$ is the Poincar\'e map. By blowing up the singularities, the rescaled Poincar\'e maps are uniformly continuous and therefore well-defined on domains with uniformly bounded below sizes. One can refer \cite{CY,GY} for the construction of extended rescaled Poincar\'e map $P^*_T $. For the singular aperiodic chain recurrent classes, we show the extended rescaled Poincar\'e maps are associated with central models. Meanwhile, there must be chain recurrent central segments over singularities.

\medskip

Let us state our result more mathematically. Suppose $X$ is a generic $C^1$ vector field robustly without horseshoes, $\mathrm{Sing}(X)$ is the collection of singularities of $X$. For $\sigma\in \mathrm{Sing}(X)$, the chain recurrent class and saddle value of $\sigma$ are denoted by $C(\sigma)$ and $\mathrm{sv}(\sigma)$ respectively. Let us define:
\begin{align*}
G_\sigma=\{L\in PT_\sigma M:&\exists X_n\rightarrow X~in~C^1,x_n\in \mathrm{Per}(X_n)\\
&  such~that~\mathcal{O}(x_n)\hookrightarrow C(\sigma),\langle\exp^{-1}_{\sigma}(x_n) \rangle\rightarrow L\} ,
\end{align*}
and $ K_\sigma=G_\sigma\cup( C(\sigma)\setminus \mathrm{Sing}(X))$. Suppose there exists $\rho\in C(\sigma)\cap \mathrm{Sing}(X)$ such that $\mathrm{sv}(\sigma)\mathrm{sv}(\rho)<0$. Then the extended rescaled Poincar\'e map $P^*_1$ over $ K_\sigma$ is partially hyperbolic with one dimensional center. In the same way as \cite[Proposition 4.6]{Xiao}, there exist a finite cover $\ell:\hat{K}_\sigma\rightarrow K_\sigma$ and a central model $(\hat{K}_\sigma,\hat{f} ) $ to depict the dynamics of the extended rescaled Poincar\'e map restricted to the one dimensional locally invariant central manifolds. With these conventions, our main results can be concluded as:
\begin{main}
In the central model $(\hat{K}_\sigma,\hat{f} ) $, there exists $\hat{x}\in \ell^{-1}(G_\sigma)$ and a chain recurrent central segment over $\hat{x}$.
\end{main}

In the statement of the main theorem, the central model does not have arbitrarily small trapping/repelling neighborhoods. On the other hand, the existence of chain recurrent central segments over singularities does not increase the dimension of the chain recurrent set along the center, because the zero flow speed needs to be taken into account. Therefore, differing from the nonsingular flow case, the chain recurrent central segment in this central model does not give birth to horseshoes. Thus neither aspects of the dichotomy about the central model create horseshoes. Therefore, the strategy of central model fails to eliminate the non-Lyapunov stable singular aperiodic chain recurrent classes. \emph{This implies solo using central model is insufficient to solve weak Palis conjecture in higher dimensional($\geq 4$) singular flows}.

The proof of the main theorem contains three steps. The first step is devoted to the construction of extended rescaled Poincar\'e maps(Proposition \ref{beta}). To this end, the blowup of singularity is introduced. Though there are available references about this topic, for instance \cite{CY,T}, but we are able to address the blowup construction in a more elementary way so that the construction of extended rescaled Poincar\'e map is simplified.
It is worth to remark that the novelty lies in the reduction to linear vector fields(the second step of the proof). To be more precise, we prove the extended rescaled Poincar\'e map over singularity equals the counterpart of the linearized vector field (Lemma \ref{l}). Furthermore, the linearized vector field is hyperbolic with the stable and unstable subspaces each containing a one dimensional weak direction, and one can choose a unit vector $u$ in the two dimensional center such that $\langle u \rangle\in G_\sigma$.
Finally, we show the machinery to associate a central model to the chain recurrent class. The estimation of the extended rescaled Poincar\'e maps at $\langle u \rangle$ implies the existence of chain recurrent central segment over $\ell^{-1}(\langle u \rangle) $ in the central model as we wanted.

In addition, as a byproduct of the blowup construction, we are applicable to compute the second order derivatives of extended rescaled Poincar\'e maps, which we believe has its independent interests. For example, for linear vector fields on two dimensional Euclidean space, we show that the extended rescaled Poincar\'e maps are generally nonlinear.

This work is organized as following: In section \ref{blowup}, we address the blowup construction. In section \ref{linear} we prove the main theorem, deducing that solo using central model is insufficient to solve weak Palis conjecture in higher dimensional($\geq 4$) singular flows. In the appendix, we compute the second order derivatives of the extended rescaled Poincar\'e maps of two dimensional linear vector fields.

\section{Blowup of singularities}\label{blowup}
In this section we readdress the blowup construction in a more elementary way compared to the available references, for instance \cite{CY,T}. Based on this tool, the construction of extended rescaled Poincar\'e maps in proving the main theorem is possibly simplified. Meanwhile, as a byproduct, we are able to compute the second order derivatives of the extended rescaled Poincar\'e map. This result is new and interesting as far as we are concerned so we put it in the appendix.
\subsection{Local: polar coordinate transformation}\label{local}
In this subsection we interpret the local construction of blowup of singularity as the polar coordinate transformation.

Suppose $n\in\mathbb{N}$, $n\geq 2$. Let $X$ be a a $C^1$ vector field on $\mathbb{R}^n$ with $X(0)=0$. The flow and tangent flow are denoted by $\phi_t$ and $\Phi_t$ respectively. Let us consider the polar coordinate transformation:
\begin{align*}
J:S^{n-1}\times[0,+\infty) & \rightarrow \mathbb{R}^n\\
(u,s) & \mapsto s\cdot u.
\end{align*}
\begin{lemma}\label{pullback}
There exists a continuous vector field $\tilde{X}$ on $S^{n-1}\times [0,+\infty)$ such that for any $(u,s)\in S^{n-1}\times [0,+\infty)$,
$$DJ\tilde{X}(u,s)=X(s\cdot u).$$
Meanwhile, the action of $\tilde{X}$ on $S^{n-1}\times \{0\}$ is equal to the normalization of $\Phi_t(0)$.
\end{lemma}
\begin{proof}
Let us first compute the tangent map $DJ_{(u,s)}:T_uS^{n-1}\times\mathbb{R}\rightarrow\mathbb{R}^n $. Suppose $\{e_1,\cdots,e_{n-1} \}$ is a base of $T_uS^{n-1}$, and $e_n$ is the unit vector of $\mathbb{R}$. This implies $\{e_1,\cdots,e_n \}$ and $\{e_1,\cdots,e_{n-1},u \}$ are basis of $T_uS^{n-1}\times\mathbb{R}$ and $\mathbb{R}^n$ respectively. Under these two basis, the following holds:
\begin{equation*}
DJ_{(u,s)}=\mathrm{diag}\{s,\cdots,s,1\}.
\end{equation*}
For $s\neq 0$, the vector $X(s\cdot u)$ has an orthogonal decomposition:
\begin{equation*}
X(s\cdot u)=\big(X(s\cdot u)-\big\langle X(s\cdot u),u  \big\rangle u\big)+\big\langle X(s\cdot u),u  \big\rangle u.
\end{equation*}
There exists a vector $\tilde{X}(u,s)$ on $T_uS^{n-1}\times\mathbb{R}$ such that $DJ\tilde{X}(u,s)=X(s\cdot u)$:
\begin{align}
\tilde{X}(u,s) & =\frac{X(s\cdot u)-\big\langle X(s\cdot u),u  \big\rangle u}{s}+\big\langle X(s\cdot u),u  \big\rangle e_n\\
& =\int^1_0DX(t\cdot s \cdot u)u-\big\langle DX(t\cdot s\cdot u)u,u  \big\rangle u dt +\big\langle X(s\cdot u),u  \big\rangle e_n.
\end{align}
Let us define:
\begin{equation}
\tilde{X}(u,0)=DX(0)u-\big\langle DX(0)u,u  \big\rangle u.
\end{equation}
By (2.2), the vector field $\tilde{X}$ is continuous on $S^{n-1}\times [0,+\infty)$.

Let us consider the flow of $\tilde{X}$. For $s\neq 0$,
\begin{equation}
\phi_t(s\cdot u)=\int^1_0\Phi_t(w\cdot s\cdot u)s\cdot u d w.
\end{equation}
Suppose $\Vert \phi_t(s\cdot u)\Vert=s_t,~\frac{\phi_t(s\cdot u)}{\Vert \phi_t(s\cdot u)\Vert}=u_t $. According to (2.4), the following holds:
\begin{equation}
\frac{s_t}{s}\cdot u_t=\int ^1_0\Phi_t(w\cdot s \cdot u)udw.
\end{equation}
By taking $(u,s)\rightarrow (u_0,0) $, the RHS of (2.5) tends to $\Phi_t(0)u_0$. Therefore,
\begin{align}
u_t  \rightarrow \frac{\Phi_t(0)u_0}{\Vert \Phi_t(0)u_0 \Vert},~
\frac{s_t}{s}  \rightarrow \Vert \Phi_t(0)u_0 \Vert.
\end{align}
Let us define:
\begin{align}
&\tilde{\phi}_t(u,s)=(u_t,s_t)=(\frac{\phi_t(s\cdot u)}{\Vert \phi_t(s\cdot u)\Vert},\Vert \phi_t(s\cdot u)\Vert),~s\neq 0,\\
&\tilde{\phi}_t(u,0)= (\frac{\Phi_t(0)u}{\Vert \Phi_t(0)u \Vert},0 ).
\end{align}
From (2.6), (2.7) and (2.8), one can see $ \tilde{\phi}_t$ is a continuous flow on $S^{n-1}\times [0,+\infty)$ that is tangent to  $ \tilde{X}$. Meanwhile, (2.8) indicates that on $S^{n-1}\times \{0\}$, the flow $ \tilde{\phi}_t$ is the normalization of $\Phi_t(0)$. The proof of the lemma is finished.
\end{proof}
\begin{remark}\label{2}
By (2.3), the vector $ \tilde{X}(u,0)$ is the orthogonal projection of $DX(0)u$ onto $\langle u\rangle^\perp$, and therefore the unit eigenvectors of $DX(0)$ are singularities of $ \tilde{X}$.
\end{remark}
In order to be boundaryless, let us define an equivalence relation $\sim$ on $S^{n-1}\times [0,+\infty)$ as following:
\begin{equation*}
(u,s)\sim (u,s),~(u,0)\sim (-u,0).
\end{equation*}
The quotient space $S^{n-1}\times [0,+\infty)/\sim$ is a $C^\infty$ boundaryless manifold. Meanwhile, the map $J$ induces a map from $S^{n-1}\times [0,+\infty)/\sim$ to $\mathbb{R}^n$. Let us denote it by $\hat{J}$.
\begin{remark}\label{hat}
Since $\tilde{X}(u,0)=- \tilde{X}(-u,0)$, the vector field $\tilde{X}$ induces a continuous vector field $\hat{X}$ on $S^{n-1}\times [0,+\infty)/\sim$. According to Lemma \ref{pullback}, the quotient vector field $\hat{X}$ generates a continuous flow on $S^{n-1}\times [0,+\infty)/\sim$.
\end{remark}
\subsection{Global: compactification of manifold minus singularities}
Given a $C^1$ vector field $X$ with non-degenerate singularities on the manifold $M$, the global construction of blowup of singularities is a way to compactify the manifold minus singularities.
\begin{lemma}
There exist a compact boundaryless manifold $\hat{M}$, a $C^\infty$ surjective map $\Pi:\hat{M}\rightarrow M$ and a continuous vector field $\hat{X}$ on $\hat{M}$ such that
\begin{enumerate}
\item  $\Pi|\hat{M}\setminus \Pi^{-1}(\mathrm{Sing}(X))$ is a diffeomorphism onto $ M\setminus \mathrm{Sing}(X)$, $\Pi^{-1}(M\setminus \mathrm{Sing}(X) ) $ is dense in $\hat{M}$;
\item for any $\sigma\in \mathrm{Sing}(X)$, there exists a neighborhood $U$ such that $\Pi: \Pi^{-1}(U)\rightarrow U$ is equal to $\hat{J}$ modulo coordinate charts.
\item  $D\Pi(\hat{X})=X $, $\hat{X}$ generates a continuous flow $\hat{\phi}_t$ on $\hat{M}$.
\end{enumerate}
\end{lemma}

\begin{proof}
Suppose $\mathrm{Sing}(X)=\{\sigma_1,\cdots,\sigma_k \}$. For $i=1,\cdots, k $, let $s_i>0$ be small enough such that
\begin{itemize}
\item $\exp_{\sigma_i}:B(0,s_i)\subset T_{\sigma_i }M \rightarrow U_i $ is a diffeomorphism;
\item $U_i=\exp_{\sigma_i}( B(0,s_i)),i=1,\cdots,k$ are pairwise disjoint.
\end{itemize}
Let us define:
$\hat{M}=(M\setminus \mathrm{Sing}(X))\cup PT_{\sigma_i}M\cup\cdots\cup PT_{\sigma_k}M$.
The space $\hat{M}$ is endowed a topology such that:
\begin{itemize}
\item the map $j:M\setminus \mathrm{Sing}(X)(\subset M)\rightarrow M\setminus \mathrm{Sing}(X)(\subset\hat{M})$ with $j(x)=x$ is a homeomorphism;
\item for $i=1,\cdots,k$, the map $I_i:PT_{\sigma_i}M\rightarrow PT_{\sigma_i}M\subset \hat{M}$ such that $I_i(\langle u\rangle)=\langle DX(\sigma_i)u\rangle $ is an embedding.
\end{itemize}
Let us define a map $\Pi:\hat{M}\rightarrow M$ such that
\begin{align*}
\pi(x)&=x,~x\in M\setminus\mathrm{Sing} (X),\\
\pi(\langle u\rangle)&=\sigma_i,~u\in T^1_{\sigma_i }M,~i=1,\cdots,k.
\end{align*}
The nondegeneracy of $\sigma_i$ implies the neighborhood $V_i=U_i\setminus  \{\sigma_i\}\cup PT_{\sigma_i }M$ of $PT_{\sigma_i }M $ is homeomorphic to $T^1_{\sigma_i }M\times [0,s_i)\diagup (v,0)\sim (-v,0)$ by the following map:
\begin{align*}
\varphi_i:T^1_{\sigma_i }M\times [0,s_i)\diagup \sim & \rightarrow V_i\\
(u,s) & \mapsto \exp_{\sigma_i }(s\cdot u),s>0,\\
(u,0) & \mapsto \langle u\rangle.
\end{align*}
Therefore $\Pi:V_i\rightarrow U_i $  is equal to $\exp_{\sigma_i}\circ \hat{J}\circ\varphi_i^{-1} $ for $i=1,\cdots,k$, and Item 2 of this lemma holds.

On the other hand, the coordinate charts of $M\setminus \mathrm{Sing}(X)$ are $C^\infty$ consistent with $\{ \varphi_i\}$. Hence $\hat{M}$ is a $C^\infty$ compact manifold under these coordinate charts and $\phi_i,i=1,\cdots,k$. Item 1 is deduced directly from the choice of the topology of $\hat{M} $. Item 3 follows from Remark \ref{hat} and Item 2. The proof of this lemma is finished.

\end{proof}

Suppose $\xi:TM\rightarrow M$ is the tangent bundle. Let $\Pi^\ast(\xi):\Pi^\ast(TM)\rightarrow \tilde{M}$ be the pullback of $\xi$ by $\Pi$:
\begin{displaymath}
\xymatrix{
		\Pi^\ast(TM)\ar[d]^{\Pi^\ast(\xi)} \ar[r]^{}    & TM \ar[d]^{\xi}\\
		\hat{M}   \ar[r]^{\Pi}                                        &M}
\end{displaymath}
By the choice of the topology on $\hat{M}$, $\Pi^\ast(\xi)$ admits a continuous line field $\mathcal{L}$ such that
\begin{align}
\mathcal{L}_x=\langle X(x)\rangle~for~x\in M\setminus \mathrm{Sing}(X),~\mathcal{L}_{\langle u\rangle }=\langle DX(\sigma_i)u\rangle.
\end{align}
Let us recall the definition of the normal bundle $\mathcal{N}$ of $X$:
\begin{equation*}
\mathcal{N}=\{v\in T_xM:x\in M\setminus \mathrm{Sing}(X),\langle v,X(x) \rangle=0\}.
\end{equation*}
Let $\hat{N}$ be the orthogonal complement of $\mathcal{L}$. Then the restriction of $\hat{N}$ to $M\setminus \mathrm{Sing}(X) $, namely $\hat{N}_{M\setminus \mathrm{Sing}(X)}$, is isomorphic to $\mathcal{N}$ by (2.9).
\begin{remark}
The definition of $\mathcal{L}$ in (2.9) implies the Nash blowup of singularities in \cite{LGW} is homeomorphic to our blowup construction. With $\mathcal{L}$ as reference lines, the generalized linear poincar\'e flow introduced in \cite{LGW} is well-defined in $\hat{N}$ as following:
\begin{equation}
\psi_t:\hat{N}\rightarrow\hat{N},~\psi_t(v)=\pi(\Phi_t(u) ),
\end{equation}
with $\pi$ the orthogonal projection from $\Pi^\ast(TM)$ to $\hat{N}$.
\end{remark}
\section{Proof of the main theorem}\label{linear}
The proof of the main theorem contains three steps. The first step is the construction of extended rescaled Poincar\'e map. It is not new, but simple and important for the construction of central model. Second, we show the reduction to linear vector fields. Third, we show the existence of central model, and in this central model there must be chain recurrent central segments over singularities through estimations of extended rescaled Poincar\'e map.
\subsection{Extended rescaled Poincar\'e map}
It is proved that the rescaled Poincar\'e map are defined on domains with uniformly bounded below sizes and can be compactified in \cite{CY,GY,Wx}. To be more precise,
\begin{proposition} \label{beta}
For any $T> 0$, there exists $\beta>0$ such that the rescaled Poincar\'e map $ \mathcal{P}^*_T$ is well-defined on the normal bundle $\mathcal{N}(\beta)$ and can be extended to a continuous map $P^*_T:\hat{N}(\beta)\rightarrow\hat{N}$.
\end{proposition}
The elementary way of addressing blowup construction can clarify the construction of extended rescaled Poincar\'e maps. Meanwhile, Proposition \ref{beta} is crucial for the construction of central model in proving the main theorem. So let us give a proof of this Proposition.
\begin{proof}[Proof of Proposition \ref{beta}]
Let us consider the neighborhood of a singularity and modulo the local coordinate transformations. Suppose $X$ is a $C^1$ vector field on $\mathbb{R}^d$ such that $X(0)=0$. For $x\neq 0,t\in \mathbb{R}$, $0<r\ll 1$ and $y\in N_x(r)$, let $\tau+t=\tau(t,x,y)+t$ be the first time for $y$ to reach $N_{\phi_t(x)}$. The rescaled Poincar\'e map satisfies:
\begin{align*}
\mathcal{P}^*_{t,x}(y)=&\frac{1}{\Vert X(\phi_t(x))\Vert}\mathcal{P}_{t,x}(\Vert X(x) \Vert y)\\
=&\frac{1}{\Vert X(\phi_t(x))\Vert}\exp^{-1}_{\phi_t(x) }\circ P_{t,x}\circ \exp_x(\Vert X(x) \Vert y )\\
=&\frac{1}{\Vert X(\phi_t(x))\Vert}\exp^{-1}_{\phi_t(x) }\circ \phi_{\tau+t}\circ\exp_x( \Vert X(x) \Vert y)\\
=&\frac{1}{\Vert X(\phi_t(x))\Vert}\big(\phi_{\tau+t}(x+\Vert X(x) \Vert y)-\phi_t(x) \big).
\end{align*}
For $(u,s)\in S^{d-1}\times (0,+\infty)$, $\tau\in \mathbb{R}$, $x=s\cdot u$ and $y\in N_x$, let us define:
\begin{equation*}
F(t,u,s,\tau,y ) =\frac{1}{\Vert X(\phi_t(x))\Vert}\big(\phi_{\tau+t}(x+\Vert X(x) \Vert y)-\phi_t(x) \big),
\end{equation*}
\begin{equation}
=\frac{\Vert X(x) \Vert}{\Vert X(\phi_t(x))\Vert}\int^1_0d\phi_{\tau+t}\big(s\cdot u+w\cdot\Vert X(x) \Vert y \big)ydw+\frac{\phi_{\tau+t}(x)-\phi_t(x)}{\Vert X(\phi_t(x))\Vert}.
\end{equation}
For $s=0$, let us define $F(t,u,0,\tau,y)$ such that
\begin{equation}
F(t,u,0,\tau,y)=\frac{\Vert DX(0)u \Vert}{\Vert DX(0)d\phi_t(0)u \Vert}d\phi_{\tau+t }(0)y+\frac{d\phi_{\tau+t }(0)u-d\phi_t(0)u}{\Vert DX(0)d\phi_t(0)u \Vert}.
\end{equation}
Equations (3.1) and (3.2) imply $F$ is a continuous map.
For $s\neq 0$, the first order derivatives of $F$ are:
\begin{align}
\frac{\partial F}{\partial y}(t,u,s,\tau,y)=&\frac{ \Vert X(x) \Vert}{\Vert X(\phi_t(x))\Vert}d\phi_{\tau+t}\big(x+\Vert X(x) \Vert y\big),\\
\frac{\partial F}{\partial \tau}(t,u,s,\tau,y)=&\frac{1}{\Vert X(\phi_t(x))\Vert}X\big(\phi_{\tau+t}(x+\Vert X(x) \Vert y) \big)\\
=&\frac{\Vert X(x) \Vert}{\Vert X(\phi_t(x))\Vert}\int^1_0DX\big(\phi_{\tau+t}(x+w\cdot\Vert X(x) \Vert y) \big)\\
&\cdot d\phi_{\tau+t}\big(x+w\cdot\Vert X(x) \Vert y\big)ydw+\frac{X(\phi_{\tau+t}(x) )}{\Vert X(\phi_t(x))\Vert}.
\end{align}
Let us define:
\begin{align}
\frac{\partial F}{\partial \tau}(t,u,0,\tau,y )=&\frac{\Vert DX(0)u \Vert}{\Vert DX(0)d\phi_t(0)u \Vert}DX(0)d\phi_{\tau+t }(0)y+\frac{DX(0)d\phi_{\tau+t}(0)u}{\Vert DX(0)d\phi_t(0)u \Vert},\\
\frac{\partial F}{\partial y}(t,u,0,\tau,y)=&\frac{\Vert DX(0)u \Vert}{\Vert DX(0)d\phi_t(0)u \Vert}d\phi_{\tau+t }(0),
\end{align}
By (3.3)-(3.8), the first order derivatives $\frac{\partial F}{\partial \tau}$ and $\frac{\partial F}{\partial y} $ are continuous.
Let us define $H(t,u,s,\tau,y)$ as following:
\begin{equation}
H(t,u,s,\tau,y)=\big\langle F(t,u,s,\tau,y ),\frac{X\big(\phi_t(x)\big)}{\big\Vert X\big(\phi_t(x)\big)\big\Vert} \big\rangle.
\end{equation}
From (3.2) and (3.4) one can see:
\begin{align*}
H(t,u,s,0,0)&=0,\\
\frac{\partial H}{\partial \tau}(t,u,s,0,0)&=\big\langle \frac{X\big(\phi_t(x)\big)}{\big\Vert X\big(\phi_t(x)\big)\big\Vert} ,\frac{X\big(\phi_t(x)\big)}{\big\Vert X\big(\phi_t(x)\big)\big\Vert}\big\rangle=1.
\end{align*}
By the Explicit Function Theorem, there exists a map $\tau=\tau(t,u,s,y)$ such that
\begin{equation}
H\big(t,u,s,\tau( t,u,s,y),y\big)=0.
\end{equation}
Meanwhile, the following holds:
\begin{itemize}
\item $\frac{\partial\tau}{\partial y }$ is continuous;
\item for fixed $t=T$ and $S>0$, there exists $\alpha>0$ such that for any $s\leq S$ the sizes the domains of $\tau=\tau(T,u,s,\cdot)$ are greater than $\alpha$.
\end{itemize}
By (3.10), the time for $y$ to reach $N_{\phi_t(x)}$ for $s\neq 0$ is $t+\tau(t,u,s,y)$, and therefore
\begin{equation}
F\big(t,u,s,\tau( t,u,s,y),y\big)=\mathcal{P}^*_{t,x}(y).
\end{equation}
Let us define:
\begin{equation}
\tilde{P}^*(t,u,s,y)=F\big(t,u,s,\tau( t,u,s,y),y\big).
\end{equation}
By (3.11) and (3.12), $\tilde{P}^*(t,u,s,y)=\mathcal{P}^*_{t,x}(y)$ for $s\neq 0$. By (3.2) and (3.9), one has $\tau( t,u,0,y)=\tau( t,-u,0,-y) $, and therefore
\begin{equation*}
F\big(t,u,0,\tau( t,u,0,y),y\big)=-F\big(t,-u,0,\tau( t,-u,0,-y),-y\big).
\end{equation*}
According to the definition of $\hat{N}$ and (3.12), $\tilde{P}^*$ induces a map $P^* $ in the neighborhood of $\hat{N}_{\Pi^{-1}(0)}(\alpha)$. Meanwhile, (3.11) implies $P^*$ is the extension of $\mathcal{P}^*_T$ nearby the singularity.

On the other hand, given a regular point $x$ and for any $y$ close to $x$, the domain of the rescaled Poincar\'e maps $ \mathcal{P}^*_{T,y}$ has uniformly bounded below sizes. Therefore there exists $0<\beta\leq \alpha$ such that the rescaled Poincar\'e map $ \mathcal{P}^*_T$ is well-defined on $\mathcal{N}(\beta)$. So we have proved that the rescaled Poincar\'e map $ \mathcal{P}^*_T$ can be extended to a continuous map $P^*_T:\hat{N}(\beta)\rightarrow\hat{N}$.
\end{proof}

\begin{definition}\label{d}
The generalized rescaled linear Poincar\'e flow $\psi_t^*:\hat{N}\rightarrow\hat{N}$ is defined as:
\begin{equation}
\psi_t^*(v)=\frac{\psi_t(v)}{\Vert \Phi_t|\mathcal{L}_x \Vert}~for~x\in\hat{N}_x.
\end{equation}
\end{definition}
\begin{lemma}\label{r}
The derivative of the extended rescaled Poincar\'e map $P^*_t$ is equal to the generalized rescaled linear Poincar\'e map $\psi^*_t$.
\end{lemma}
\begin{proof}
Let us compute directly from (3.12). For $s\neq 0$,
\begin{align*}
\frac{\partial P^*}{\partial y}(t,u,s,0)(v)=&\frac{\partial F}{\partial y}(t,u,s,0,0)v+\frac{\partial F}{\partial \tau}(t,u,s,0,0 )\big\langle\frac{\partial \tau}{\partial y}(0),v\big\rangle\\
=&\frac{\big\Vert X(x)\big\Vert}{\big\Vert X\big(\phi_t(x) \big)\big\Vert}\Phi_t(x)v+\frac{X\big(\phi_t(x)\big)}{\big\Vert X\big(\phi_t(x)\big)\big\Vert}\big\langle\frac{\partial \tau}{\partial y}(0),v\big\rangle\\
=&\frac{\big\Vert X(x)\big\Vert}{\big\Vert X\big(\phi_t(x) \big)\big\Vert}\big(\Phi_t(x)v+\frac{\big\langle\frac{\partial \tau}{\partial y}(0),v\big\rangle}{\big\Vert X(x)\big\Vert} X\big(\phi_t(x)\big)\big)\in \mathcal{N}_{\phi_t(x)} \\
=&\frac{\big\Vert X(x)\big\Vert}{\big\Vert X\big(\phi_t(x) \big)\big\Vert}\pi\big(\Phi_t(x)(v)\big)=\psi^*_t(v),
\end{align*}
Since $\frac{\partial P^*}{\partial y} $ is continuous, one has for $s=0$
\begin{align}
\frac{\partial P^*}{\partial y}(t,u,0,0)(v)= \frac{\big\Vert DX(0)u\big\Vert}{\big\Vert DX(0)\Phi_t(0)u \big\Vert}\pi\big(\Phi_t(0)(v)\big)=\psi^*_t(v).
\end{align}
The proof of this lemma is finished.
\end{proof}
\subsection{Reduction to linear vector fields}
As stressed in the introduction, we want to eliminate singular aperiodic chain recurrent classes of vector fields robustly without horseshoes. In fact, these chain recurrent classes are usually not Lyapunov stable, the dimension is greater than 3.

Suppose $\dim M\geq 4$, $X$ is a $C^1$ generic vector field robustly without horseshoes. For $\sigma\in \mathrm{Sing}(X)$, $T_\sigma M=E^s_\sigma\oplus E^u_\sigma $ is the hyperbolic splitting, the Lyapunov exponents are: $\lambda_1\leq \cdots\leq\lambda_i<0<\lambda_{i+1}\leq\cdots\leq \lambda_d$.
The saddle value $\mathrm{sv}(\sigma)$ is defined as: $\mathrm{sv}(\sigma)=\lambda_i+\lambda_{i+1}$. Suppose there exists $\rho\in \mathrm{Sing}(X)\cap C(\sigma)$ such that $\mathrm{sv}(\sigma)\mathrm{sv}(\rho) <0$. Let us recall the definition of $G_\sigma$ in the introduction:
\begin{align*}
G_\sigma=\{L\in PT_\sigma M:&\exists X_n\rightarrow X~in~C^1,x_n\in \mathrm{Per}(X_n)\\
&  such~that~\mathcal{O}(x_n)\hookrightarrow C(\sigma),\langle\exp^{-1}_{\sigma}(x_n) \rangle\rightarrow L\} ,
\end{align*}
and $ K_\sigma=G_\sigma\cup( C(\sigma)\setminus \mathrm{Sing}(X))$.
\begin{lemma}\label{G}
\begin{enumerate}
\item The hyperbolic splitting of $\sigma$ satisfies:
$$E^s_\sigma=E^{ss}_\sigma\oplus E^{cs}_\sigma,~E^u_\sigma=E^{cu}_\sigma\oplus E^{uu}_\sigma,$$
with $\dim E^{cs}_\sigma=\dim E^{cu}_\sigma =1$. Moreover, $G_\sigma\subset E^c_\sigma=E^{cs}_\sigma\oplus E^{cu}_\sigma $.
\item $K_\sigma$ admits a partially hyperbolic splitting with respect to the generalized rescaled linear Poincar\'e flow:
$$\hat{N}_{K_\sigma}=N^s\oplus N^c\oplus N^u,~\dim N^c=1. $$
\end{enumerate}
\end{lemma}
\begin{remark}
The definition of $G_\sigma$ and Item 1 of Lemma \ref{G} imply the periodic points of nearby vector fields whose orbits are close to $C(\sigma) $ accumulate $\sigma$ only along the two dimensional center direction.
\end{remark}
\begin{proof}
According to \cite[Lemma 3.3.4]{Zheng},
\begin{itemize}
\item the singularity $\sigma$ has a splitting $E^{ss}_\sigma\oplus E^{cs}_\sigma\oplus E^{cu}_\sigma\oplus E^{uu}_\sigma $ with $\dim(E^{cs}_\sigma)=\dim(E^{cu}_\sigma)=1 $;
\item $\dim E^{ss}_\sigma\neq 0,\dim E^{uu}_\sigma\neq 0$, $W^{ss}(\sigma)\cap C(\sigma)=\{\sigma\}$, and $W^{uu}(\sigma)\cap C(\sigma)=\{\sigma\}$;
\item $K_\sigma$ has a partially hyperbolic splitting with respect to the generalize linear Poincar\'e flow:
\begin{equation}
\hat{N}_{K_\sigma}=N^s\oplus N^c\oplus N^u,~\dim N^c=1
\end{equation}
\end{itemize}
By the same arguments as in \cite[Lemma 4.4]{LGW}, one has the following:
\begin{align*}
G_\sigma\subset( E^{ss}_\sigma\oplus E^{cs}_\sigma\oplus E^{cu}_\sigma)\cap (E^{cs}_\sigma\oplus E^{cu}_\sigma\oplus E^{uu}_\sigma)= E^{cs}_\sigma\oplus E^{cu}_\sigma.
\end{align*}
Therefore Item 1 is proved. The proof of Item 2 is based on the following claim:
\begin{claim}
$N^s$ is dominated by $\mathcal{L}_{K_\sigma}$ and $\mathcal{L}_{K_\sigma}$ is dominated by $N^u$.
\end{claim}
\begin{proof}[Proof of the claim]
For $L\in G_\sigma\subset E^c$, one has $N^s_L=E^{ss}_\sigma,~N^u_L=E^{uu}_\sigma$. Then the claim follows by similar arguments as \cite[Lemma 5.3]{LGW}.
\end{proof}
By definition \ref{d} and the claim, $N^s$/$N^u$ is contracted/expanded by the generalized rescaled linear Poincar\'e flow.
Therefore the splitting (3.15) is as wanted in the statement of Item 2.
The proof of this lemma is finished.
\end{proof}

In order to estimate the generalized rescaled Poincar\'e maps over $G_\sigma$, let us fix a local chart of $\sigma$ as in the global construction of blowing up singularities. Let us first compare the extended rescaled Poincar\'e map over a singularity with the counterpart of the linearized vector field.

Suppose $0\in \mathbb{R}^d$ corresponds to the singularity $\sigma$, $E^{ss}_\sigma,E^{cs}_\sigma,E^{cu}_\sigma,E^{uu}_\sigma$ are pairwise orthogonal, and $X(x)=Ax+f(x)$ with $f(0)=0,Df(0)=0$. Moreover, there exist $A^{ss}\in \mathrm{Gl}(i-1,\mathbb{R})$ and $A^{uu}\in \mathrm{Gl}(d-i-1,\mathbb{R})$ such that for any $x=(x^{ss},x^{cs},x^{cu},x^{uu})$,
\begin{equation}
Ax=(A^{ss}x^{ss},\lambda_i x^{cs},\lambda_{i+1} x^{cu},A^{uu}x^{uu} ).
\end{equation}
\begin{lemma}\label{l}
The extended rescaled Poincar\'e maps of $X$ over $PT_\sigma M$ are equal to the counterpart of the vector field $Y=Ax$.
\end{lemma}
\begin{proof}
Recall the extended rescaled Poincar\'e map $P^*$ satisfies:
\begin{align}
&P^*(t,u,s,y)=F\big(t,u,s,\tau( t,u,s,y),y\big),\\
& H\big(t,u,s,\tau( t,u,s,y),y\big)=0,\\
&H(t,u,s,\tau,y)=\big\langle F(t,u,s,\tau,y ),\frac{X\big(\phi_t(x)\big)}{\big\Vert X\big(\phi_t(x)\big)\big\Vert} \big\rangle .
\end{align}
From (3.2), (3.18) and (3.19), one can see $\tau(t,u,0,y )$ satisfies:
\begin{equation}
\big\langle \frac{\big\Vert DX(0)u \big\Vert}{\big\Vert DX(0)d\phi_t(0)u \big\Vert}d\phi_{\tau+t }(0)y+\frac{d\phi_{\tau+t }(0)u-d\phi_t(0)u}{\big\Vert DX(0)d\phi_t(0)u \big\Vert}, \frac{ DX(0)d\phi_t(0)u}{\big\Vert DX(0)d\phi_t(0)u \big\Vert} \big\rangle=0.
\end{equation}
Since (3.20) is independent of $f(x)$, one has $\tau(t,u,0,y )$ and therefore $P^*(t,u,0,y ) $ are also independent of $f$. The proof of Lemma \ref{l} is finished.

\end{proof}
\begin{remark}\label{k}
Lemma \ref{l} indicates that the extended rescaled Poincar\'e maps over singularity are independent of the nearby regular orbits;
\end{remark}

\subsection{Chain recurrent central segment over singularity}
Let us first show the construction of central model.
\begin{lemma}\label{c}
There exists a central model $(\hat{K}_\sigma ,\hat{f})$, a finite cover $ \ell:\hat{K}_\sigma\rightarrow K_\sigma$ and a continuous map $\alpha: \hat{K}_\sigma\times[0,+\infty)\rightarrow \hat{N}$ such that
\begin{enumerate}
\item the map $\ell$ is at most two folds, $\hat{K}_\sigma$ is chain transitive, and the derivative of $\alpha$ with respect to the second variable is continuous;
\item for any $\hat{x}\in \hat{K}_\sigma$ and $x=\ell(\hat{x})$, the center plaque $F_x=\alpha(\{\hat{x}\}\times[0,1) )\subset \hat{N}_x$ is tangent to $N^c_x $, the family $\{F_x\}_{x\in K_\sigma}$ is locally invariant under the extended rescaled Poincar\'e map $P^*_1$;
\item the map $\alpha$ semi-conjugates $\hat{f}$ and $\hat{P}^*_1 $: $\alpha\circ\hat{f}|_{\{\hat{x} \}\times[0,1)}=\hat{P} ^*_{1,x}\circ \alpha|_{\{\hat{x} \}\times[0,1)}$.
\end{enumerate}

\end{lemma}
\begin{remark}
As indicated by Lemma \ref{c}, the central model $(\hat{K}_\sigma ,\hat{f})$ depicts the dynamics of the extended rescaled Poincar\'e map $\hat{P} ^*_1$ along the one dimensional center direction $N^c$.
\end{remark}
\begin{proof}
According to Lemma \ref{r} and Item 2 of Lemma \ref{G}, the extended rescaled Poincar\'e map $P^*_1 $ is partially hyperbolic with one dimensional center. Then one can see the lemma holds by following the arguments in \cite[Proposition 4.6]{Xiao}.

\end{proof}

With the preparations, let us begin the proof of the main theorem.
\begin{proof}[Proof of main theorem]
By Item 1 of lemma \ref{G}, there exists
$$u=(0^{ss},\cos\theta,\sin\theta,0^{uu}),$$
such that $\langle u\rangle \in G_\sigma$, $\theta\neq \frac{k\pi}{2}$.
According to Item 2 of Lemma \ref{G}, one has
 $$N^c_{\langle u\rangle }=\langle v\rangle,~with~v=(0^{ss},-\sin\theta,\cos\theta,0^{uu} ) .$$
By Lemma \ref{l}, the following holds for $\lambda_i=-1,~\lambda_{i+1}=1$:
\begin{align}
\psi_t^*(v)=&\frac{\big\Vert Au\big\Vert}{\big\Vert Ae^{tA}u \big\Vert}\big(e^{tA}v-\frac{\langle e^{tA}v,Ae^{tA}u\rangle}{\langle Ae^{tA}u,Ae^{tA}u\rangle}Ae^{tA}u  \big)\\
=&\frac{\cos^2\theta-\sin^2\theta}{\big(\sqrt{e^{-2t}\cos^2\theta+e^{2t}\sin^2\theta}\big)^3}(0^{ss}, e^t\sin\theta,e^{-t}\cos\theta,0^{uu}).
\end{align}
From (3.22) one can see
\begin{equation}
\lim_{t\rightarrow\pm\infty}\psi_t^*(v)= 0 .
\end{equation}
Therefore $N^c_{\langle u\rangle }$ is contracted exponentially by the generalized rescaled linear Poincar\'e flow as $t\rightarrow\pm\infty$. By Item 2 of Lemma \ref{c} and (3.14), the center plaque $F_{\langle u\rangle}$ is contracted exponentially by both the extended rescaled Poincar\'e map $P^*_1$ and its inverse. From Item 3 of Lemma \ref{c}, for any $\hat{x}\in\ell^{-1}(\langle u\rangle)$, the fiber $\{ \hat{x}\}\times[0,1]$ contains a segment $\gamma$ that is contracted by both $\hat{f}$ and $\hat{f}^{-1}$. The segment $\gamma$ is in the same chain recurrent class as $\hat{K}_\sigma $ and therefore is a chain recurrent central segment. The proof of the main theorem is finished.

\end{proof}
\begin{remark}
The assumption $\lambda_i=-1,~\lambda_{i+1}=1 $ is not essential in the proof of the main theorem, but it simplifies the computation. In fact, (3.23) holds for any $\lambda_i<0,~\lambda_{i+1}>0 $.
\end{remark}
\subsection{Central model isinsufficient to solve weak Palis conjecture in higher dimensional singular flow }
The main theorem illustrates central model is insufficient to eliminate non-Lyapunov stable singular aperiodic chain recurrent classes. Let us explain it explicitly.

Suppose $(\hat{K},\hat{f}) $ is a central model and the base $\hat{K}\times \{0\} $ is chain transitive. The creation of horseshoe via center model is based on the following dichotomy with the flavor of Conley theory:
\begin{itemize}
\item either there exists chain recurrent central segment;
\item or the base $ \hat{K}\times \{0\}$ admits arbitrarily small attracting/repelling neighborhoods.
\end{itemize}

But in the central model $(\hat{K}_\sigma ,\hat{f})$ given by Lemma \ref{c}, neither aspects of the dichotomy create horseshoes.

First, the central model given by Lemma \ref{c} admits no chain recurrent central segments over regular orbits. To be more precise,
\begin{proposition}\label{n}
In the central model $(\hat{K}_\sigma ,\hat{f})$ of Lemma \ref{c}, for any $\hat{x}\in\ell^{-1}( C(\sigma)\setminus \mathrm{Sing}(X))$ and any $0<a<1$, the segment $\{\hat{x}\}\times [0,a]$ is not a chain recurrent central segment.
\end{proposition}
\begin{proof}
Suppose $\hat{x}\in\ell^{-1}( C(\sigma)\setminus \mathrm{Sing}(X))$ such that $\{\hat{x}\}\times [0,a]$ is a chain recurrent central segment, $x=\ell(\hat{x})$, $\hat{\gamma}=\alpha(\{\hat{x}\}\times [0,a] )$, and
\begin{equation}
\gamma=\exp_x(\Vert X(x)\Vert \hat{\gamma }).
\end{equation}
By Lemma \ref{c}, one has
\begin{equation}
T_x\gamma=N_x,~\gamma\subset C(\sigma).
\end{equation}
Let $\mathcal{O}(p)$ be a periodic orbit close to $C(\sigma )$  and passing nearby $\gamma$.
Let us show that $\mathcal{O}(p)$ and $\gamma$ form a heteroclinic cycle by Figure 1.

As indicated by the figure, there exists a pseudo-orbit from $\mathcal{O}(p)$ to the strong unstable manifold of $\mathcal{O}(p)$, reaching the strong stable manifold of a point $x\in \gamma$, then going inside $C(\sigma)$ from $x$ to certain point $y\in \gamma$, going on along the strong unstable manifold of $y$, until reaching the strong stable manifold of $\mathcal{O}(p) $, and along the strong stable manifold of $\mathcal{O}(p)$ back to $\mathcal{O}(p)$.

Therefore $\mathcal{O}(p) $ is contained in the same chain recurrent class as the segment $\gamma$. Meanwhile, (3.25) implies $\mathcal{O}(p)\subset C(\sigma) $, a contradiction to the assumption of $C(\sigma)$ being aperiodic. Therefore there exist no chain recurrent central segments over regular points in the central model $(\hat{K}_\sigma ,\hat{f})$.
\begin{figure}[h]
\centering
\begin{tikzpicture}[scale=4]

\draw (0,0) -- (1,0) -- (1,1.2) -- (0,1.2) -- (0,0);
\draw (0,0) -- (-0.8,-0.3) -- (-0.8, 0.9) -- (0,1.2) -- (0,0);
\draw[directed] (0,0.7) -- (0.8,0.7);
\draw[directed] (0.8,0.7) -- (0.4,0.2);
\draw[directed] (0.4,0.2) -- (-0.6,0.2);
\draw[directed] (-0.6,0.2) -- (0,.425);

\node[above] at (-0.8,0.9) {$ss$};
\node[above] at (0,1.2) {$\gamma$};
\node[above] at (0.1,0.7) {$y$};
\node[below,right] at (0.4,0.2) {$\mathcal{O}(p)$};
\node[above,left] at (0,.44) {$x$};
\node[right] at (1,1.2) {$uu$};

\end{tikzpicture}
\caption{Chain recurrent central segment and heteroclinic cycle}\label{fig1}
\end{figure}
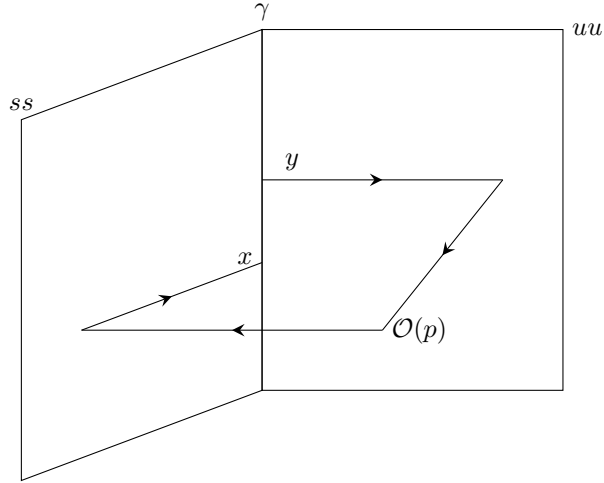

\end{proof}

Second, as one can infer from (3.24), once the zero flow speed is taken into account, the existence of chain recurrent central segment in the main theorem does not increase the dimension of the chain recurrent class along the center direction. Therefore in the central model $(\hat{K}_\sigma,\hat{f} ) $, the chain recurrent central segment guaranteed by the main theorem does not create horseshoes.

Third, the dichotomy about central model implies $(\hat{K}_\sigma,\hat{f})$ does not have arbitrarily small trapping/repelling neighborhoods. Therefore the other mechanism for the birth of horseshoes by central model does not work.

So we come to the conclusion: The strategy of central model does not work to eliminate generic aperiodic chain recurrent classes with singularities whose saddle values have different signs. Differing from $C^1$ diffeomorphisms and nonsingular flows, solo using central model isbe insufficient to solve weak Palis conjecture in higher dimensional ($\geq4$) singular flows.

\section*{Appendix}
As indicated by Remark \ref{k}, the extended rescaled Poincar\'e maps over singularity are determined exclusively by the linearized vector field of the singularity. Therefore we believe it is interesting to calculate the extended rescaled Poincar\'e map of linear vector fields. We compute upto the second order derivatives. It turns out the extended rescaled Poincar\'e maps of two dimensional linear vector fields are generally nonlinear.
\subsection*{A1. Extended rescaled Poincar\'e map under moving orthogonal frame}
For $A\in \mathrm{Gl}(2,\mathbb{R})$, the solution of the differential equation
\begin{equation*}
 \dot{x}=Ax,
\end{equation*}
is $\phi_t(x)=e^{tA}x$. Assume $u=(x_1,x_2)$ is a unit vector, $y\in \mathbb{R}$, $(Au )^\perp $ is a rotation of $Au$ by $\frac{\pi}{2}$. Let us define the \emph{extended rescaled Poincar\'e map under moving orthogonal frame} by the following equation:
\begin{align*}
P^*\big(t,u,0,y\frac{(Au )^\perp}{\Vert Au \Vert }\big)=F^*_{t,u}(y)\frac{(Ae^{tA}u )^\perp}{\Vert (Ae^{tA}u )^\perp \Vert }.
\end{align*}
The extended rescaled Poincar\'e map $P^*$ satisfies:
\begin{align*}
P^*\big(t,u,0,y\frac{(Au )^\perp}{\Vert Au \Vert }\big)=\frac{e^{\big(\tau(y)+t\big)A }y(Au )^\perp}{\Vert Ae^{tA}u \Vert}+\frac{e^{\big(\tau(y)+t\big)A }u-e^{tA}u}{\Vert Ae^{tA}u \Vert},
\end{align*}
with a $\tau=\tau(y)$ such that
\begin{equation*}
\big\langle \frac{e^{\big(\tau(y)+t\big)A }y(Au )^\perp}{\Vert Ae^{tA}u \Vert}+\frac{e^{\big(\tau(y)+t\big)A }u-e^{tA}u}{\Vert Ae^{tA}u \Vert},\frac{Ae^{tA}u}{\Vert Ae^{tA}u \Vert}  \big\rangle =0.
\end{equation*}
\begin{pro}\emph{
The second order derivative of the extended rescaled Poincar\'e map satisfies:
\begin{align*}
\frac{d^2F^*_{t,u}}{dy^2}(0)=2\frac{\big\langle Ae^{tA }(Au )^\perp,(Ae^{tA}u )^\perp \big\rangle}{\big\langle Ae^{tA}u,Ae^{tA}u \big\rangle}\frac{d\tau}{dy}(0)+\frac{\big\langle A^2e^{tA }u,(Ae^{tA}u )^\perp \big\rangle}{\big\langle Ae^{tA}u,Ae^{tA}u \big\rangle}\big(\frac{d\tau}{dy}(0) \big)^2.
\end{align*}}
\end{pro}
\begin{proof}
Let us define $H(t,u,\tau,y)$ by:
\begin{equation*}
H(t,u,\tau,y)=\big\langle \frac{e^{(\tau+t)A }y(Au )^\perp}{\Vert Ae^{tA}u \Vert}+\frac{e^{(\tau+t)A }u-e^{tA}u}{\Vert Ae^{tA}u \Vert},\frac{Ae^{tA}u}{\Vert Ae^{tA}u \Vert}  \big\rangle.
\end{equation*}
Then $H(t,u,0,0)=0 $. Meanwhile, $\frac{d\tau}{dy}(0)  $ satisfies:
\begin{align*}
\frac{d\tau}{dy}(0)&=-\frac{\frac{\partial H}{\partial  y}(t,u,0,0)}{\frac{\partial H}{\partial  \tau}(t,u,0,0)}=-\frac{\big\langle e^{tA }(Au )^\perp , Ae^{tA}u \big\rangle}{\big\langle Ae^{tA}u,Ae^{tA}u \big\rangle}.
\end{align*}
We can see the following holds:
\begin{align*}
F^*_{t,u}(y)=&\big\langle P^*\big(t,u,0,y\frac{(Au )^\perp}{\Vert Au \Vert }\big) ,\frac{(Ae^{tA}u )^\perp}{\Vert (Ae^{tA}u )^\perp \Vert } \big\rangle=\big\langle Q\big(t,u,\tau(y),y\big),\frac{(Ae^{tA}u )^\perp}{\Vert (Ae^{tA}u )^\perp \Vert } \big\rangle,
\end{align*}
with $Q(t,u,\tau,y)$ defined as following
\begin{align*}
Q(t,u,\tau,y)=\frac{e^{(\tau+t)A }y(Au )^\perp}{\Vert Ae^{tA}u \Vert}+\frac{e^{(\tau+t)A }u-e^{tA}u}{\Vert Ae^{tA}u \Vert}.
\end{align*}
Therefore the second order derivative of $F^*_{t,u}(y) $ satisfies:
\begin{align*}
\frac{d^2F^*_{t,u}}{dy^2}(0)=&\big\langle 2\frac{\partial^2 Q(t,u,0,0)}{\partial y\partial\tau}\frac{d\tau}{dy}(0),\frac{(Ae^{tA}u )^\perp}{\Vert (Ae^{tA}u )^\perp \Vert }  \big\rangle\\
&+\big\langle \frac{\partial^2 Q(t,u,0,0)}{\partial \tau^2}(0)\big(\frac{d\tau}{dy}(0) \big)^2 ,\frac{(Ae^{tA}u )^\perp}{\Vert (Ae^{tA}u )^\perp \Vert }  \big\rangle\\
=&2\frac{\big\langle Ae^{tA }(Au )^\perp,(Ae^{tA}u )^\perp \big\rangle}{\big\langle Ae^{tA}u,Ae^{tA}u \big\rangle}\frac{d\tau}{dy}(0)+\frac{\big\langle A^2e^{tA }u,(Ae^{tA}u )^\perp \big\rangle}{\big\langle Ae^{tA}u,Ae^{tA}u \big\rangle}\big(\frac{d\tau}{dy}(0) \big)^2.
\end{align*}
\end{proof}
\begin{rmk}
To see whether the second order derivatives vanish, we need to compute the following four inner products:
\begin{align*}
&\big\langle e^{tA }(Au )^\perp , Ae^{tA}u \big\rangle,~\big\langle Ae^{tA }(Au )^\perp , (Ae^{tA}u )^\perp \big\rangle,\\
&\big\langle A^2e^{tA }u , (Ae^{tA}u )^\perp \big\rangle,~\big\langle Ae^{tA}u,Ae^{tA}u \big\rangle.
\end{align*}

\end{rmk}

\subsection*{A2. The non-vanishing second order derivatives }
Suppose $A\in \mathrm{Gl(2,\mathbb{R})}$. Then $A$ is similar to one of the following three types:
(1):$\left(
                \begin{array}{cc}
                  \lambda_1 & 0 \\
                  0 & \lambda_2 \\
                \end{array}
              \right)$,
              (2):$\left(
              \begin{array}{cc}
              \lambda &0\\
              1& \lambda\\
              \end{array}
              \right)$,
              (3):$\left(
              \begin{array}{cc}
              \alpha & -\beta \\
              \beta & \alpha\\
              \end{array}
              \right)$.
($\lambda_1\neq \lambda_2,\lambda>0,\alpha^2+\beta^2>0$.)
\begin{pro}\label{4.1}\emph{
\begin{enumerate}
\item[(1)] If $A$ is of the third type, $F^*_{t,u}$ is a linear function;
\item[(2)] If $A$ is of the first type, the second order derivative $\frac{d^2F^*_{t,u}}{dy^2}(0) $ does not vanish (except when $u$ is an eigenvector), and therefore, $F^*_{t,u} $ is nonlinear.
\item[(3)]\label{elliptic} If $A$ is of the second type, $F^*_{t,u} $ is generally nonlinear.
\end{enumerate}}
\end{pro}
\begin{rmk}
Since a unit eigenvector is a singularity of the extended flow by Remark \ref{2}, the extended rescaled Poincar\'e map over the eigenvector is the identity. The sense we mean by 'generally' in item (3) will be illustrated in the proof.
\end{rmk}
\begin{proof}
\textbf{The third type:}
Suppose $A=\left(
\begin{array}{cc}
\alpha &\-\beta\\
\beta & \alpha\\
\end{array}
\right)$, $x=r(\cos \theta,\sin \theta)$. Then $e^{tA}x= re^{t\alpha}\big(\cos(\theta+t\beta ),\sin(\theta+t\beta)\big)$.
This implies that $\phi_t=e^{tA}$ is conformal. Therefore, for any unit vector $u$, the orthogonal section to $Ax$ at $u$ is mapped by $\phi_t$ to the orthogonal section at $e^{tA}u$. Consequently, one has $\tau(t,u,y)=0$ and the following holds:
\begin{align*}
F^*_{t,u}(y)=&\big\langle \frac{e^{tA }y(Au )^\perp}{\Vert Ae^{tA}u \Vert},\frac{(Ae^{tA}u )^\perp}{\Vert (Ae^{tA}u )^\perp \Vert } \big\rangle\\
=&y\big\langle \frac{(Ae^{tA }u )^\perp}{\Vert Ae^{tA}u \Vert},\frac{(Ae^{tA}u )^\perp}{\Vert (Ae^{tA}u )^\perp \Vert } \big\rangle\\
=&y.
\end{align*}
So we have shown that the extended rescaled Poincar\'e map under moving frame $ F^*_{t,u}$ is linear if the singularity is a focus.

\textbf{The first type:}
Suppose $A=\left(
                \begin{array}{cc}
                  \lambda_1 & 0 \\
                  0 & \lambda_2 \\
                \end{array}
              \right) $, $ \lambda_1\neq\lambda_2,\lambda_1\lambda_2\neq 0 $.
For any unit vector $u=(x_1,x_2)$,
the following equations hold:
\begin{align}
\frac{d^2F^*_{t,u}}{dy^2}(0)=&2\frac{\big\langle Ae^{tA }(Au )^\perp , (Ae^{tA}u )^\perp \big\rangle}{\big\langle Ae^{tA}u,Ae^{tA}u \big\rangle}\frac{d\tau}{dy}(0)+\frac{\big\langle A^2e^{tA }u , (Ae^{tA}u )^\perp \big\rangle}{\big\langle Ae^{tA}u,Ae^{tA}u \big\rangle}\big(\frac{d\tau}{dy}(0) \big)^2\\
=&\frac{\lambda_1^2\lambda_2^2x_1x_2e^{t(\lambda_1+\lambda_2) }(e^{ 2t\lambda_1}-e^{2t\lambda_2 } )}{(\lambda_1^2x_1^2e^{2t\lambda_1 }+\lambda_2^2x_2^2e^{2t\lambda_2 } )^3}\\
&\cdot\big(S(\lambda_1,x_1,\lambda_2,x_2 )e^{ 2t\lambda_1}+ S(\lambda_2,x_2,\lambda_1,x_1 )e^{ 2t\lambda_2}\big),
\end{align}
with $S(\lambda_1,x_1,\lambda_2,x_2 )=(2\lambda_1^2x_1^2+\lambda_2^2x_2^2+\lambda_1\lambda_2x_2^2 )\lambda_1x_1^2$.

Let us define: $R(\lambda_1,x_1,\lambda_2,x_2 )=2\lambda_1^2x_1^2+\lambda_2^2x_2^2+\lambda_1\lambda_2x_2^2$.
For $u=(x_1,x_2)$ such that $x_1x_2\neq 0$, the equation $\lambda_1\neq\lambda_2$ implies $R(\lambda_1,x_1,\lambda_2,x_2 ) $ and $R(\lambda_2,x_2,\lambda_1,x_1 ) $ can not vanish simultaneously. By (3.27) and (3.28), the second order derivative of the extended rescaled Poincar\'e map $F^*_{t,u}$ does not vanish.

\textbf{The second type:}
Suppose $A=\left(
              \begin{array}{cc}
              \lambda &0\\
              1& \lambda\\
              \end{array}
              \right)$, $\lambda\neq 0$.
Let $u=(x_1,x_2) \in S^1$.
The second order derivative of the extended rescaled Poincar\'e map is
\begin{align*}
\frac{d^2F^*_{t,u}}{dy^2}(0)=&2\frac{\big\langle Ae^{tA }(Au )^\perp , (Ae^{tA}u )^\perp \big\rangle}{\big\langle Ae^{tA}u,Ae^{tA}u \big\rangle}\frac{d\tau}{dy}(0)+\frac{\big\langle A^2e^{tA }u , (Ae^{tA}u )^\perp \big\rangle}{\big\langle Ae^{tA}u,Ae^{tA}u \big\rangle}\big(\frac{d\tau}{dy}(0) \big)^2\\
=&\frac{\lambda^2\big(\lambda^2tx_1^2-\lambda^2tx_2^2-2\lambda tx_1x_2-tx_1^2-\lambda x_1(\lambda x_2+x_1 )t^2 \big)}{\big(\lambda^2x_1^2+(\lambda x_2+\lambda tx_1+x_1)^2 \big)^3}\cdot P,
\end{align*}
with the coefficient of $t^2$ in the polynomial $P $ equal to $- \lambda x_1^2\big((2\lambda^2+1 )x_1^2+3\lambda x_1 x_2+2\lambda^2x_2^2 \big)$.
The coefficient of the highest order term of $t$ in $\frac{d^2F^*_{t,u}}{dy^2}(0)$ is
\begin{equation}
\frac{\lambda^4 x_1^3(\lambda x_2+x_1 )\big((2\lambda^2+1 )x_1^2+3\lambda x_1 x_2+2\lambda^2x_2^2 \big)}{\big(\lambda^2x_1^2+(\lambda x_2+x_1)^2 \big)^3 }.
\end{equation}
Since the polynomial $ (2\lambda^2+1 )x_1^2+3\lambda x_1 x_2+2\lambda^2x_2^2 $ is positive definite, (3.29) does not vanish if $x_1\neq 0,\lambda x_2+x_1\neq 0 $. For $x_1=0$, $u$ is an eigenvector. For $ \lambda x_2+x_1= 0$, the computation is involved so we prefer not to check out whether $\frac{d^2F^*_{t,u}}{dy^2}(0)$ vanishes. In conclusion, the extended rescaled Poincar\'e map $F^*_{t,u}$ is generally nonlinear. Therefore the proof of the proposition is finished.
\end{proof}
\section*{Acknowledgements}
The authors express their deep gratitude to Prof. Lan Wen and Prof. Shaobo Gan for useful discussions and encouragements.
Y. Zhang is partially supported by the NSFC grant 11701200 and Hubei Providence Youth Science and Technology Scholar funding.

\end{document}